\documentclass{article}
\usepackage{amsmath,amssymb,amsthm}
\usepackage{mathtools}

\usepackage[all]{xy}

\usepackage{amsfonts,mathrsfs,graphicx,multirow,latexsym}
\usepackage[mathscr]{euscript}
\usepackage{float}

\usepackage{cellspace}
\usepackage[export]{adjustbox}

\usepackage{makecell}

\newtheoremstyle%
{custom}%
{}
{}
{}
{}
{}
{.}
{ }
{\thmname{}
\thmnumber{}%
\thmnote{\bfseries #3}}%

\newtheoremstyle%
{Theorem}%
{}%
{}%
{\itshape}%
{}%
{}%
{.}%
{ }%
{\thmname{\bfseries #1}%
\thmnumber{\;\bfseries #2}%
\thmnote{\;(\bfseries #3)}}%

\theoremstyle{Theorem}
\newtheorem{theorem}{Theorem}[section]

\newtheorem{lemma}{Lemma}[section]
\newtheorem{prop}{Proposition}[section]

\theoremstyle{definition}
\newtheorem{definition}{Definition}[section]

\theoremstyle{remark}

\theoremstyle{custom}

\usepackage[colorinlistoftodos]{todonotes}
\usepackage[colorlinks=true, allcolors=blue]{hyperref}

\title{Enumerating Diagonalizable Matrices over $\mathbb{Z}_{p^k}$}
\author{Catherine Falvey, Heewon Hah, William Sheppard, Brian Sittinger,\\ Rico Vicente}
\date{\vspace{-5ex}}

\begin{document}
\maketitle

\begin{abstract}
Although a good portion of elementary linear algebra concerns itself with matrices over a field such as $\mathbb{R}$ or $\mathbb{C}$, many combinatorial problems naturally surface when we instead work with matrices over a finite field. As some recent work has been done in these areas, we turn our attention to the problem of enumerating  the square matrices with entries in $\mathbb{Z}_{p^k}$ that are diagonalizable over $\mathbb{Z}_{p^k}$. This turns out to be significantly more nontrivial than its finite field counterpart  due to the presence of zero divisors in $\mathbb{Z}_{p^k}$.
\end{abstract}

\section{Introduction}

A classic problem in linear algebra concerns whether a matrix $A \in M_n(K)$ (where $K$ is a field) is diagonalizable: There exists an invertible matrix $P \in GL_n(K)$ and a diagonal matrix $D \in M_n(K)$ such that $A = PDP^{-1}$. It is known that if $A$ is diagonalizable, then $D$ is unique up to the order of its diagonal elements. Besides being useful for computing functions of matrices (and therefore often giving a solution to a system of linear differential equations), this problem has applications in the representation of quadratic forms.

\vspace{.1 in}

If we consider $M_n(K)$ when $K$ is a finite field, one natural problem is to enumerate $\text{Eig}_n(K)$, the set of $n \times n$ matrices over $K$ whose $n$ eigenvalues, counting multiplicity, are in $K$. Olsavsky \cite{Olsavsky} initiated this line of inquiry, and determined that for any prime $p$,
$$|\text{Eig}_2(\mathbb{F}_p)| = \frac{1}{2} \Big(p^4 + 2p^3 - p^2\Big).$$
\noindent More recently, Kaylor and Offner \cite{Kaylor} gave a procedure to enumerate $\text{Eig}_n(\mathbb{F}_q)$, thereby extending Olsavsky's work for any $n$ and any finite field $\mathbb{F}_q$.

\vspace{.1 in}

Inspired by these works, we turn our attention to $n \times n$ matrices over $\mathbb{Z}_{p^k}$, where $p$ is a prime and $k$ is a positive integer. More specifically, we investigate the problem about enumerating $\text{Diag}_n(\mathbb{Z}_{p^k})$, the set of $n \times n$ diagonalizable matrices over $\mathbb{Z}_{p^k}$. This is significantly more involved when $k \geq 2$, and many of the difficulties arise from having to carefully consider the zero divisors of $\mathbb{Z}_{p^k}$, namely any integral multiple of $p$.

\vspace{.1 in}

In Section 2, we review the pertinent definitions and notations for working with matrices over commutative rings. Most notably, we give a crucial theorem that essentially states that a diagonalizable matrix over $\mathbb{Z}_{p^k}$ is unique up to the ordering of its diagonal entries. In Section 3, we give the basic procedure for enumerating $\text{Diag}_n(\mathbb{Z}_{p^k})$ and apply it to the case where $n=2$ in Section 4.
In order to deal with the cases where $n \geq 3$ in a systematic manner, we introduce to any diagonal matrix an associated weighted graph in Section 5 that allows us to find $|\text{Diag}_3(\mathbb{Z}_{p^k})|$ and $|\text{Diag}_4(\mathbb{Z}_{p^k})|$ in Sections 6 and 7, respectively. In the final sections, we use our work to find the proportion of matrices that are diagonalizable over $\mathbb{Z}_{p^k}$ and conclude by giving ideas for future research based on the ideas in this article. As far as we understand, all results and definitions from Proposition 3.1 in Section 3 onward are original.

\section{Background}

In this section, we give some definitions from matrix theory over rings that allow us to extend some notions of matrices from elementary linear algebra to those having entries in $\mathbb{Z}_{p^k}$. For the following definitions, we let $R$ denote a commutative ring with unity. For further details, we refer the interested reader to \cite{Brown}.

To fix some notation, let $M_n(R)$ denote the set of $n \times n$ matrices with entries in $R$. The classic definitions of matrix addition and multiplication as well as determinants generalize in $M_n(R)$ in the expected manner. In general, $M_n(R)$ forms a non-commutative ring with unity $I_n$, the matrix with 1s on its main diagonal and 0s elsewhere.

Next, we let $GL_n(R)$ denote the set of invertible matrices in $M_n(R)$; that is,
$$GL_n(R) = \{A \in M_n(R) \, : \, AB = BA = I_n \text{ for some } B \in M_n(R)\}.$$

\noindent Note that $GL_n(R)$ forms a group under matrix multiplication and has alternate characterization
$$GL_n(R) = \{A \in M_n(R) \, : \, \det A \in R^*\},$$
\noindent where $R^*$ denotes the group of units in $R$. Observe that when $R$ is a field $K$, we have $K^* = K \backslash \{0\}$; thus we retrieve the classic fact for invertible matrices over $K$. For this article, we are specifically interested in the case when $R = \mathbb{Z}_{p^k}$ where $p$ is prime and $k \in \mathbb{N}$. Then, $$GL_n(\mathbb{Z}_{p^k}) = \{A \in M_n(\mathbb{Z}_{p^k}) \, | \, \det A \not\equiv 0 \bmod p\};$$ 
\noindent in other words, we can think of an invertible matrix with entries in $\mathbb{Z}_{p^k}$ as having a determinant not divisible by $p$.

\begin{definition}
We say that $A \in M_n(R)$ is \textbf{diagonalizable over $R$} if $A$ is similar to a diagonal matrix $D \in M_n(R)$; that is, $A=PDP^{-1}$ for some $P \in GL_n(R)$.
\end{definition}

Recall that any diagonalizable matrix over a field is similar to a distinct diagonal matrix that is unique up to ordering of its diagonal entries. Since $\mathbb{Z}_{p^k}$ is \emph{not} a field whenever $k \geq 2$, we now give a generalization of this key result to matrices over $\mathbb{Z}_{p^k}$. This provides a foundational result that allows us to use the methods from \cite{Kaylor} to enumerate diagonalizable matrices over $\mathbb{Z}_{p^k}$. Although we originally came up for a proof for this result, the following elegant proof was suggested to the authors by an anonymous MathOverflow user; see \cite{User}.

\begin{theorem} \label{thm:DDT} Any diagonalizable matrix over $\mathbb{Z}_{p^k}$ is similar to exactly one diagonal matrix that is unique up to ordering of its diagonal entries.
\end{theorem}

\begin{proof}
Suppose that $D, D' \in M_n(\mathbb{Z}_{p^k})$ are diagonal matrices such that $D' = PDP^{-1}$ for some $P \in GL_n(\mathbb{Z}_{p^k})$. Writing $D = \text{diag}(d_1, \dots , d_n)$, $D' = \text{diag}(d'_1, \dots , d'_n)$, and $P = (p_{ij})$, we see that $D' = PDP^{-1}$ rewritten as $PD = D' P$ yields $p_{ij} d_i = p_{ij} d'_j$ for all $i, j$.

\vspace{.1 in}

Since $P \in GL_n(\mathbb{Z}_{p^k})$, we know that $\det{P} \in \mathbb{Z}_{p^k}^*$, and thus $\det{P} \not\equiv 0 \bmod p$. However, since $\det{P} = \sum_{\sigma \in S_n} (-1)^{\text{sgn}(\sigma)} \prod_{i} p_{i, \sigma(i)}$, and the set of non-units in $\mathbb{Z}_{p^k}$ (which is precisely the subset of elements congruent to 0 mod $p$) is additively closed, there exists $\sigma \in S_n$ such that $\prod_{i} p_{i, \sigma(i)} \in \mathbb{Z}_{p^k}^*$ and thus $p_{i,\sigma(i)} \in \mathbb{Z}_{p^k}^*$ for all $i$.

\vspace{.1 in}

Then for this choice of $\sigma$, it follows that $p_{i,\sigma(i)} d_i = p_{i,\sigma(i)} d'_{\sigma(i)}$ for each $i$, and since $p_{i,\sigma(i)} \in \mathbb{Z}_{p^k}^*$, we deduce that $d_i = d'_{\sigma(i)}$ for each $i$. In other words, $\sigma$ is a permutation of the diagonal entries of $D$ and $D'$, giving us the desired result. 
\end{proof}

\vspace{.1 in}

\noindent \textbf{Remark:} Theorem \ref{thm:DDT} does not extend to $\mathbb{Z}_m$ for a modulus $m$ with more than one prime factor. As an example from \cite{Brown}, the matrix
$\begin{pmatrix}
    2 & 3 \\
    4 & 3
\end{pmatrix} \in M_2(\mathbb{Z}_6)$ has two distinct diagonalizations
$$\begin{pmatrix}
    1 & 3 \\
    2 & 1
\end{pmatrix}
\begin{pmatrix}
    2 & 0 \\
    0 & 3
\end{pmatrix}
\begin{pmatrix}
    1 & 3 \\
    2 & 1
\end{pmatrix}^{-1} =
\begin{pmatrix}
    1 & 3 \\
    5 & 2
\end{pmatrix}
\begin{pmatrix}
    5 & 0 \\
    0 & 0
\end{pmatrix}
\begin{pmatrix}
    1 & 3 \\
    5 & 2
\end{pmatrix}^{-1}.$$
The resulting diagonal matrices are thus similar over $\mathbb{Z}_6$ although their diagonal entries are not rearrangements of one another.

\section{How to determine \texorpdfstring{$|\text{Diag}_n(\mathbb{Z}_{p^k})|$}{TEXT}}

In this section, we give a procedure that allows us to determine $|\text{Diag}_n(\mathbb{Z}_{p^k})|$, the number of matrices in $M_n(\mathbb{Z}_{p^k})$ that are diagonalizable over $\mathbb{Z}_{p^k}$.
The main idea is to use a generalization of a lemma from Kaylor (Lemma 3.1 in \cite{Kaylor}). Before stating it, we first fix some notation in the following definition.

\begin{definition}

Let $R$ be a commutative ring with 1, and fix $A \in M_n(R)$.

\begin{itemize}

\item The \textbf{similarity (conjugacy) class} of $A$, denoted by $S(A)$, is the set of matrices similar to $A$:
$$S(A) = \{B\in M_n(R) \, : \, B=PAP^{-1}
\text{ for some } P \in GL_n(R)\}.$$

\item The \textbf{centralizer} of $A$, denoted by $C(A)$, is the set of invertible matrices that commute with $A$:
$$C(A) = \lbrace P \in GL_n(R) \, : \, PA=AP \rbrace.$$

\end{itemize}

\end{definition}

\noindent Note that $P \in C(A)$ if and only if $A=PAP^{-1}$, and moreover $C(A)$ is a subgroup of $GL_n(R)$.

\begin{lemma} \label{lemma:counting} Let $R$ be a finite commutative ring. For any $A \in M_n(R)$, we have $\displaystyle \vert S(A)\vert = \frac{\vert GL_n(R)\vert }{\vert C(A)\vert}.$ 
\end{lemma}

\begin{proof}
This is proved verbatim as Lemma 3.1 in \cite{Kaylor} upon replacing a finite field with a finite commutative ring. Alternatively, this is a direct consequence of the Orbit-Stabilizer Theorem where $GL_n(R)$ is acting on $M_n(R)$ via conjugation.
\end{proof}


To see how this helps us in $M_n(\mathbb{Z}_{p^k})$, recall by Theorem \ref{thm:DDT} that the similarity class of a given diagonalizable matrix can be represented by a unique diagonal matrix (up to ordering of diagonal entries). Therefore, we can enumerate $\text{Diag}_n(\mathbb{Z}_{p^k})$ by first enumerating the diagonal matrices in $M_n(\mathbb{Z}_{p^k})$ and then counting how many matrices in $M_n(\mathbb{Z}_{p^k})$ are similar to a given diagonal matrix. Then, Lemma \ref{lemma:counting} yields
\begin{equation}\label{eq:1}
|\text{Diag}_n(\mathbb{Z}_{p^k})| = \sum_{D \in M_n(\mathbb{Z}_{p^k})} |S(D)| = \sum_{D \in M_n(\mathbb{Z}_{p^k})}
\frac{\vert GL_n(\mathbb{Z}_{p^k})\vert }{\vert C(D)\vert},
\end{equation}
where it is understood that each diagonal matrix $D$ represents a distinct similarity class of diagonal matrices. Observe that diagonal matrices having the same diagonal entries up to order belong to the same similarity class and are counted as different matrices when computing the size of their similarity class.

First, we give a formula for $\vert GL_n(\mathbb{Z}_{p^k}) \vert$. As this seems to be surprisingly not well-known, we state and give a self-contained proof of this result inspired by \cite{Bollman} (for a generalization, see \cite{Han}).

\begin{lemma}
$\vert GL_n(\mathbb{Z}_{p^k})\vert = p^{n^2(k-1)} \displaystyle \prod_{l=1}^{n} (p^n - p^{l-1}).$
\end{lemma}

\begin{proof}
First, we compute $|GL_n(\mathbb{Z}_p)|$ by enumerating the possible columns of its matrices. For $A \in GL_n(\mathbb{Z}_p)$, there are $p^n - 1$ choices for the first column of $A$, as the zero column vector is never linearly independent. Next, we fix $l \in \{2, 3, \dots, n\}$. After having chosen the first $(l-1)$ columns, there are $(p^n - 1) - (p^{l-1} - 1) = p^n - p^{l-1}$ choices for the $l$-th column, because we want these $l$ columns to be linearly independent over $\mathbb{Z}_p$ (and there are $p$ multiples for each of the first $(l-1)$ columns). Therefore, we conclude that
$$\vert GL_n(\mathbb{Z}_{p})\vert = \displaystyle \prod_{l=1}^{n} (p^n - p^{l-1}).$$

Hereafter, we assume that $k \geq 2$. Consider the mapping $\psi : M_n(\mathbb{Z}_{p^k}) \rightarrow M_n(\mathbb{Z}_{p})$ defined by $\psi(A) = A\bmod p $; note that $\psi$ is a well-defined (due to $p \mid p^k$) surjective ring homomorphism. Moreover, since ker$\;\psi = \{A \in M_n(\mathbb{Z}_{p^k}) \, : \, \psi(A) = 0\bmod p\}$ (so that every entry in such a matrix is divisible by $p$), we deduce that $|\text{ker}\;\psi| = (p^k / p)^{n^2} = p^{(k-1)n^2}$.

\vspace{.1 in}

Then, restricting $\psi$ to the respective groups of invertible matrices, the First Isomorphism Theorem yields
$${GL_n(\mathbb{Z}_{p^k})} / {\ker\;\psi} \cong\; GL_n(\mathbb{Z}_p).$$

\noindent Therefore, we conclude that
$$\vert GL_n(\mathbb{Z}_{p^k})\vert = |\ker\psi| \cdot |GL_n(\mathbb{Z}_{p})| = p^{n^2(k-1)} \displaystyle \prod_{l=1}^{n} (p^n - p^{l-1}).$$
\end{proof}

We next turn our attention to the problem of enumerating the centralizer of a diagonal matrix in $\mathbb{Z}_{p^k}$. 

\begin{prop}\label{thm:centralizer}
Let $D \in M_n(\mathbb{Z}_{p^k})$ be a diagonal matrix whose distinct diagonal entries $\lambda_1, \dots, \lambda_g$ have multiplicities $m_1, \dots, m_g$, respectively. Then,
$$|C(D)| = \Big(\prod_{i = 1}^g |GL_{m_i}(\mathbb{Z}_{p^k})|\Big) \cdot \Big( \prod_{j = 2}^g \prod_{i = 1}^{j-1}  p^{2m_im_jl_{ij}}\Big),$$
where $l_{ij}$ is the non-negative integer satisfying $p^{l_{ij}} \mid\mid (\lambda_i - \lambda_j)$ for each $i$ and $j$; that is, 
$$\lambda_i - \lambda_j = rp^{l_{ij}} \text{ for some } r \in \mathbb{Z}_{p^{k-l_{ij}}}^*.$$
\end{prop}

\begin{proof}
Assume without loss of generality that all matching diagonal entries of $D$ are grouped together; that is, we can think of each $\lambda_i$ with multiplicity $m_i$ as having its own $m_i \times m_i$ diagonal block of the form $\lambda_i I_{m_i}$ within $D$.

\vspace{.1 in}

To find the centralizer of $D$, we need to account for all $A \in GL_n(\mathbb{Z}_{p^k})$ such that $AD = DA$. Writing $A = (A_{ij})$, where $A_{ij}$ is an $m_i \times m_j$ block, computing the necessary products and equating like entries yields $$\lambda_i A_{ij} = \lambda_j A_{ij}.$$

\noindent If $i \neq j$, then $(\lambda_i - \lambda_j) A_{ij} \equiv 0 \bmod p^k$. Therefore, $A_{ij} \equiv 0 \bmod p^{k - l_{ij}}$, and thus $A_{ij} \equiv 0 \bmod p$. Observe that this gives $p^{l_{ij}}$ possible values for each entry in $A_{ij}$ (and similarly for those in $A_{ji}$).

\vspace{.1 in}

Therefore, $A$ is congruent to a block diagonal matrix modulo $p$ with blocks $A_{ii}$ having dimensions $m_i \times m_i$ for each $i \in \{1, \dots, g\}$. Finally since $A \in GL_n(\mathbb{Z}_{p^k})$, this means that each $A_{ii} \in GL_{m_i}(\mathbb{Z}_{p^k})$. With this last observation, the formula for $|C(D)|$ now follows immediately.
\end{proof}

Proposition \ref{thm:centralizer} motivates the following classification of diagonal matrices in $\mathbb{Z}_{p^k}$.

\begin{definition}
Let $D \in M_n(\mathbb{Z}_{p^k})$ be a diagonal matrix whose distinct diagonal entries $\lambda_1, \dots, \lambda_g$ have multiplicities $m_1, \dots, m_g$, respectively. The \textbf{type} of $D$ is given by the following two quantities:
\begin{itemize}
\item The partition $n = m_1 + \dots + m_g$
\item The set $\{l_{ij}\}$ indexed over all $1 \leq i < j \leq g$, where $p^{l_{ij}} \mid\mid (\lambda_j - \lambda_i)$.
\end{itemize}

\noindent Then we say that two diagonal matrices $D, D' \in M_n(\mathbb{Z}_{p^k})$ have the \textbf{same type} if and only if $D$ and $D'$ share the same partition of $n$, and there exists a permutation $\sigma \in S_n$ such that
$l_{ij} = l'_{\sigma(i)\sigma(j)}$ for all $1 \leq i < j \leq g$. We denote the set of all distinct types of diagonal $n \times n$ matrices by $\mathcal{T}(n)$.
\end{definition}

\noindent \textbf{Example:} Consider the following three diagonal matrices from $M_3(\mathbb{Z}_8)$:
$$D_1 = \begin{pmatrix} 1 & 0 & 0\\ 0 & 2 & 0\\0 & 0 & 3\end{pmatrix},\, D_2 = \begin{pmatrix} 1 & 0 & 0\\ 0 & 1 & 0\\0 & 0 & 5\end{pmatrix}, \, D_3 = \begin{pmatrix} 1 & 0 & 0 \\ 0 & 1 & 0\\0 & 0 & 3 \end{pmatrix},\, D_4 = \begin{pmatrix} 7 & 0 & 0 \\ 0 & 5 & 0\\0 & 0 & 7 \end{pmatrix}.$$

\noindent Since $D_1$ has partition $1 + 1 + 1$, while $D_2$, $D_3$, and $D_4$ have the partition $2 + 1$, $D_1$ does not have the same type as any of $D_2$, $D_3$, and $D_4$. Moreover, $D_2$ and $D_3$ do not have the same type, because $2^2 \mid\mid(5 - 1)$, while $2^1 \mid\mid(3 - 1)$. However, $D_3$ and $D_4$ have the same type, because they share the same partition $2+1$ and $2^1$ exactly divides both $3-1$ and $7-5$.

\vspace{.1 in}

It is easy to verify that if $D$ and $D'$ are two $n \times n$ diagonal matrices of the same type, then $|C(D)| = |C(D')|$ and thus $|S(D)| = |S(D')|$. Consequently for any type $T$, define $c(T)$ and $s(T)$ by $c(T) = |C(D)|$ and $s(T) = |S(D)|$ where $D$ is any matrix of type $T$. Then, letting $t(T)$ denote the number of diagonal matrices (up to permutations of the diagonal entries) having type $T$, we can rewrite (\ref{eq:1}) as

\begin{equation} \label{eq:2}
|\text{Diag}_n(\mathbb{Z}_{p^k})| 
= \sum_{T \in \mathcal{T}(n)} t(T) \, \frac{\vert GL_n(\mathbb{Z}_{p^k})\vert }{c(T)}.
\end{equation}


\section{Enumerating the \texorpdfstring{$2 \times 2$}{TEXT} Diagonalizable Matrices}

We now illustrate our procedure for determining the value of $\vert \text{Diag}_2(\mathbb{Z}_{p^k}) \vert$. 

\begin{theorem}
The number of $2 \times 2$ matrices with entries in $\mathbb{Z}_{p^k}$ that are diagonalizable over $\mathbb{Z}_{p^k}$ is
$$\vert \emph{Diag}_2(\mathbb{Z}_{p^k}) \vert = p^k + \dfrac{p^{k+1}(p^2-1)(p^{3k}-1)}{2(p^3-1)}.$$
\end{theorem}

\begin{proof}

In order to find $\vert \text{Diag}_2(\mathbb{Z}_{p^k}) \vert$, we need to enumerate all of the $2 \times 2$ diagonal matrix types. First of all, there are two possible partitions of $2$, namely $2$ and $1+1$. The trivial partition yields one distinct type of diagonal matrices 
$$T_1 = \Big\{\begin{pmatrix}
    \lambda & 0 \\
    0 & \lambda
\end{pmatrix} \; : \; \lambda \in \mathbb{Z}_{p^k} \Big\},$$
\noindent which consists of the $2 \times 2$ scalar matrices. Since there are $p^k$ choices for $\lambda$, we have $t(T_1) = p^k$. Moreover $c(T_1) = |GL_2(\mathbb{Z}_{p^k})|$, because any invertible matrix commutes with a scalar matrix. 

\vspace{.1 in}

The nontrivial partition $2 = 1 + 1$ yields the remaining $k$ distinct types of matrices that we index by $i \in \{0, 1, \dots , k-1\}$:
$$T_2^{(i)} = \Big\{\begin{pmatrix} \lambda_1 & 0 \\ 0 & \lambda _2
\end{pmatrix} \; : \; p^i \; || \; (\lambda_1-\lambda_2) \Big\}.$$

\noindent Fix $i \in \{0, 1, \dots , k-1\}$; we now enumerate $t(T_2^{(i)})$ and $c(T_2^{(i)})$. For $t(T_2^{(i)})$, we first observe that there are $p^k$ choices for $\lambda_1$. To find the number of choices for $\lambda_2$, observe that $\lambda_1-\lambda_2 \equiv rp^i \bmod p^k$ for some unique $r \in (\mathbb{Z}_{p^{k-i}})^*$. Hence, there are $\phi(p^{k-i})$ choices for $r$ and thus for $\lambda_2$. (As a reminder, $\phi$ denotes the Euler phi function, and $\phi(p^l) = p^{l-1}(p-1)$.) Since swapping $\lambda_1$ and $\lambda_2$ does not change the similarity class of the diagonal matrix, we conclude that
$$t(T_2^{(i)})=\dfrac{p^k \phi (p^{k-i})}{2!}.$$
\noindent Next, applying Proposition \ref{thm:centralizer} yields $c(T_2^{(i)}) = p^{2i} \phi(p^k)^2.$ 

\vspace{.1 in}

Finally, we use (\ref{eq:2}) to enumerate the $2 \times 2$ diagonal matrices and conclude that
\begin{align*}
\vert\text{Diag}_2(\mathbb{Z}_{p^k})\vert &= 
t(T_1) \frac{\vert GL_n(\mathbb{Z}_{p^k})\vert }{c(T_1)} + \sum_{i=0}^{k-1} t(T_2^{(i)}) 
\frac{\vert GL_n(\mathbb{Z}_{p^k})\vert }{c(T_2^{(i)})}\\
       & = p^k + \dfrac{p^k}{2} \cdot \dfrac{p^{4(k-1)}(p^2-1)(p^2-p)}{\phi(p^k)^2} \sum_{i=0}^{k-1} \dfrac{\phi(p^{k-i})}{p^{2i}} \\
       & = p^k + \dfrac{p^k}{2} \cdot \dfrac{p^{4(k-1)}(p^2-1)(p^2-p)}{(p^{k-1} (p-1))^2} \sum_{i=0}^{k-1} \dfrac{p^{k-i-1} (p-1)}{p^{2i}} \\
       & = p^k + \dfrac{p^{4k-2}(p^2-1)}{2} \sum_{i=0}^{k-1} \dfrac{1}{p^{3i}} \\
       & = p^k + \dfrac{p^{4k-2}(p^2-1)}{2} \cdot \frac{1 - p^{-3k}}{1 - p^{-3}}, \text{ using the geometric series}\\
       & = p^k + \dfrac{p^{k+1}(p^2-1)(p^{3k}-1)}{2(p^3-1)}.
\end{align*}
\end{proof}

\noindent \textbf{Remarks}: Observe that in the case where $k = 1$, the formula reduces to $\frac{1}{2}(p^4 - p^2 + p)$, which can be found at the end of Section 3 in Kaylor \cite{Kaylor} after you remove the contributions from the $2 \times 2$ Jordan block case. Moreover, for the diagonal matrix types corresponding to the nontrivial partition and $i \geq 1$, we are dealing with differences of diagonal entries yielding zero divisors in $\mathbb{Z}_{p^k}$; these scenarios never occur when $k = 1$ because $\mathbb{Z}_p$ is a field. 

\section{Enumerating \texorpdfstring{$n \times n$}{TEXT} Diagonal Matrices of a Given Type}

\subsection{Representing a Diagonal Matrix with a Valuation Graph}

As we increase the value of $n$, the enumeration of $n \times n$ diagonalizable matrices over $\mathbb{Z}_{p^k}$ becomes more involved, because the number of distinct types becomes increasingly difficult to catalog. The difficulties come both from the powers of $p$ dividing the differences of the diagonal entries of the matrix as well as the increasing number of partitions of $n$. In order to aid us in classifying diagonal matrices into distinct types, we introduce an associated graph to help visualize these scenarios.

\vspace{.1 in}

Let $D \in M_n(\mathbb{Z}_{p^k})$ be diagonal with distinct diagonal entries $\lambda_1, \dots, \lambda_g \in \mathbb{Z}_{p^k}$. Ordering the elements in $\mathbb{Z}_{p^k}$ by $0 < 1 < 2 < \dots < p^k - 1$, we can assume without loss of generality that $\lambda_1 < \lambda_2 < \dots < \lambda_g$ (since $D$ is similar to such a matrix by using a suitable permutation matrix as the change of basis matrix). Associated to $D$, we define its associated weighted complete graph $G_D$ (abbreviated as $G$ when no ambiguity can arise) as follows: We label its $g$ vertices with the diagonal entries $\lambda_1, \lambda_2, \dots , \lambda_g$, and given the edge between the vertices $\lambda_i$ and $\lambda_j$, we define its weight $l_{ij}$ as the unique non-negative integer satisfying $p^{l_{ij}} \mid\mid (\lambda_i - \lambda_j)$.

\begin{definition}
Let $D \in M_n(\mathbb{Z}_{p^k})$ be diagonal. We call the weighted complete graph $G$ associated to $D$ as constructed above the \textbf{valuation graph}
of $D$.
\end{definition}
\bigskip

\noindent The following fundamental property of such graphs justifies why we call these valuation graphs.

\begin{prop}
\textbf{(Triangle Inequality)} \label{thm:triangleinequality} Let $G$ be a valuation graph. Given vertices $\lambda_a$, $\lambda_b$, and $\lambda_c$ in $G$ and edges $E_{ab}$, $E_{ac}$, and $E_{bc}$, the weights satisfy $l_{bc} \geq \min \{l_{ab}, l_{ac}\}$. 
In particular, $l_{bc} = \min \{l_{ab}, l_{ac}\}$ if $l_{ab} \neq l_{ac}$.
\end{prop}

\begin{proof} By hypothesis, we know that $l_{ab}$ and $l_{ac}$ are the biggest non-negative integers satisfying $$\lambda_a - \lambda_b = rp^{l_{ab}} \text{ and } \lambda_a - \lambda_c = sp^{l_{ac}} \text{ for some } r, s \in \mathbb{Z}_{p^k}^*.$$

\noindent Without loss of generality, assume that $l_{ab} \geq l_{ac}$. Then, we obtain
$$\lambda_b - \lambda_c = (\lambda_a - \lambda_c) - (\lambda_a - \lambda_b) = p^{l_{ac}} (s - r p^{l_{ab} - l_{ac}}).$$

\noindent If $l_{ab} > l_{ac}$, then $(s - r p^{l_{ab} - l_{ac}}) \in \mathbb{Z}_{p^k}^*$, and if $l_{ab} = l_{ac}$ then $s-r$ may or may not be a zero divisor in $\mathbb{Z}_{p^k}$. The claim now immediately follows.
\end{proof}

Observe that since the valuation graph arises from a diagonal matrix in $M_n(\mathbb{Z}_{p^k})$, it is clear that its weights can only attain integral values between 0 and $k-1$ inclusive. In fact, we can give another restriction on the possible values of its weights.

\begin{lemma}\label{thm:number_of_weights}
A valuation graph $G$ on $g$ vertices has no more than $g-1$ weights.
\end{lemma}

\begin{proof}
We prove this by induction on the number of vertices $g$. This claim is true for $g = 2$, because such a graph has exactly one weight. Next, we assume that the claim is true for any valuation graph on $g$ vertices, and consider a valuation graph $G$ with vertices $\lambda_1, \dots,  \lambda_{g+1}$. By the inductive hypothesis, the valuation subgraph $H$ of $G$ with vertices $\lambda_1, \dots, \lambda_g$ has no more than $g-1$ weights. It remains to consider the weights of the edges from these vertices to the remaining vertex $\lambda_{g+1}$. If none of these edges have any of the $g-1$ weights of $H$, then we are done. Otherwise, suppose that one of these edges (call it $E$) has an additional weight. Then for any edge $E'$ other than $E$ that has $\lambda_{g+1}$ as a vertex, the Triangle Inequality (Prop. \ref{thm:triangleinequality}) implies that $E'$ has no new weight. Hence, $G$ has no more than $(g-1)+1 = g$ weights as required, and this completes the inductive step.






\end{proof}

We know that for any diagonal matrix $D \in M_n(\mathbb{Z}_{p^k})$, its valuation graph $G$ satisfies the Triangle Inequality. Moreover, any complete graph on $n$ vertices satisfying the Triangle Inequality necessarily corresponds to a collection of diagonal matrices with distinct diagonal entries in $M_n(\mathbb{Z}_{p^k})$ as long as there are at most $n-1$ weights and the maximal weight is at most $k-1$. Moreover, such a graph also corresponds to a collection of diagonal matrices with non-distinct diagonal entries in $M_N(\mathbb{Z}_{p^k})$ where $N$ is the sum of these multiplicities.








\subsection{Enumerating Diagonalizable Matrices with a Given Valuation Graph} 

Throughout this section, we assume that the diagonal matrix in $M_n(\mathbb{Z}_{p^k})$ has distinct diagonal entries. Given its valuation graph $G$, we construct a specific kind of spanning tree that will aid us in enumerating the diagonal matrices in $M_n(\mathbb{Z}_{p^k})$ having valuation graph $G$. In a sense, such a spanning tree concisely shows the dependencies among the diagonal entries of a given diagonal matrix. 

\begin{prop}
Given a diagonal matrix $D \in M_n(\mathbb{Z}_{p^k})$ with distinct diagonal entries having valuation graph $G$, there exists a spanning tree $T \subset G$ from which we can uniquely reconstruct $G$. We call $T$ a \textbf{permissible spanning tree} of $G$.
\end{prop}

\begin{proof} 
Suppose that $G$ is a valuation graph on $n$ vertices with $r$ distinct weights $a_1, a_2, \ldots , a_r$ listed in increasing order. In order to construct a permissible spanning tree for $G$, we consider the following construction.

\vspace{.1 in}

For each weight $a_i$ with $1 \leq i \leq r$, define $G_{a_i}$ to be the subgraph of $G$ consisting of the edges with weight \emph{at most} $a_i$ along with their respective vertices. From the definition of a weight, we immediately see that $G_{a_1} \supseteq G_{a_2} \supseteq \dots \supseteq G_{a_r}$. Moreover, Prop. \ref{thm:triangleinequality}  implies that each connected component of $G_{a_i}$ is a complete subgraph of $G$. 

\vspace{.1 in}

To use these subgraphs to construct a permissible spanning tree for $G$, we start with the edges in $G_{a_r}$. For each connected component of $G_{a_r}$, we select a spanning tree and include all of their edges into the edge set $E$. Next, we consider the edges in $G_{a_{r-1}}$. For each connected component of $G_{a_{r-1}}$, we select a spanning tree that includes the spanning tree from the previous step. We inductively repeat this process until we have added any pertinent edges from $G_{a_1}$. (Note that since $G_{a_1}$ contains only one connected component, $T$ must also be connected.) The result is a desired permissible spanning tree $T$ for our valuation graph $G$. 

\vspace{.1 in}

Next, we show how to uniquely reconstruct the valuation graph $G$ from $T$. To aid in this procedure, we say that \textit{completing edge} of two edges $e_1,e_2$ in $G$ that share a vertex is the edge $e_3$ which forms a complete graph $K_3$ with $e_1$ and $e_2$.

\vspace{.1 in}

Start by looking at the edges having the largest weight $a_r$ in $T$. If two edges with weight $a_r$ share a vertex, then their completing edge in $G$ must also have weight $a_r$ by the maximality of $a_r$. Upon completing this procedure, there can be no other edges in $G$ of weight $a_r$, as this would violate the construction of $T$. 

\vspace{.1 in}

Next consider the edges having weight $a_{r-1}$ (if they exist). For any two edges of weight $a_{r-1}$ that share a vertex, their completing edge must have weight $a_{r-1}$ or $a_r$ by the Triangle Inequality. If the completing edge had weight $a_r$, then we have already included this edge from the previous step. Otherwise, we conclude that the completing edge must have weight $a_{r-1}$. 

\vspace{.1 in}

Continuing this process to the lowest edge coloring $a_1$, we reconstruct $G$ as desired.
\end{proof}

We now return to the problem of enumerating diagonal $n \times n$ matrices over $\mathbb{Z}_{p^k}$ of a given type. We begin with the case that $A \in M_n(\mathbb{Z}_{p^k})$ is a diagonal matrix over $\mathbb{Z}_{p^k}$ with distinct diagonal entries. Let $G$ be its associated valuation graph with $r$ distinct weights $a_1, a_2, \dots, a_r$. 

\begin{definition}
Let $T$ be a permissible spanning tree of a valuation graph $G$. We say that a subset of edges in $T$ all with weight $a_t$ are \textbf{linked} if there exists a subtree $S$ of $T$ containing these edges such that each edge in $S$ has weight at least $a_t$.
\end{definition}

We use the notion of linked edges to partition the set of edges from our permissible tree $T$ beyond their weights as follows. Let $L^{t}$ denote the set of edges in $T$ with weight $a_t$. Then, $L^{t}$ decomposes into pairwise disjoint sets $L_1^{t}, \dots, L_{\ell(t)}^{t}$ for some positive integer $\ell(t)$, where each $L_j^{t}$ is a maximal subset of linked edges from $L^{t}$.

\begin{definition}
Let $T$ be a permissible spanning tree for a given valuation graph $G$. For a given weight $a_t$, we say that $L_1^{t}, \dots, L_{\ell(t)}^{t}$ are the \textbf{linked cells} of the weight $a_t$.
\end{definition}


\begin{theorem}\label{thm:linked}
Let $G$ be a valuation graph having $r$ distinct weights $a_1,a_2,\dots,a_r$ listed in increasing order, and let $T$ be a permissible spanning tree of $G$ with linked cells $L_j^{t}$. Then, the total number of diagonal matrix classes having distinct diagonal entries in $M_n(\mathbb{Z}_{p^k})$ with an associated valuation graph isomorphic to $G$ equals

$$\frac{p^k}{|\emph{Aut}(G)|} \cdot \prod_{t=1}^r \prod_{j=1}^{\ell(t)} \prod_{i=1}^{|L_j^{t}|} \phi_{i}(p^{k-a_t}),$$
\noindent where $\phi_{i}(p^j) = p^j - ip^{j-1}$, and $\text{Aut}(G)$ denotes the set of weighted graph automorphisms of $G$. 
\end{theorem}

\begin{proof} 
Fix a valuation graph $G$. The key idea is to consider the edges of its permissible spanning tree via linked cells, one weight at a time in descending order. Throughout the proof, we use the following convention: If an edge $E$ has vertices $\lambda_1,\lambda_2$ with $\lambda_2 > \lambda_1$, we refer to the value $\lambda_2 - \lambda_1$ as the \textit{edge difference} associated with $E$.

\vspace{.1 in}

First consider the edges in the linked cell of the maximal weight $a_r$. Without loss of generality, we start with the edges in $L_1^{r}$. Since $a_r$ is maximal, we know that $L_1^{r}$ is itself a tree. For brevity, we let $m = |L_1^{r}|$. Then, $L_1^{r}$ has $m$ edges connecting its $m+1$ vertices. We claim that there are $\prod_{i=1}^m \phi_i(p^{k-a_r})$ ways to label the values of the edge differences. 

\vspace{.1 in}

To show this, we start by picking an edge in $L_1^{r}$, and let $\lambda_1$ and $\lambda_2$ denote its vertices. Since $\lambda_2 - \lambda_1 = s_1 p^{a_r}$ for some $s_1 \in \mathbb{Z}_{p^{k-a_r}}^*$, we see that $\lambda_2 - \lambda_1$ can attain $\phi(p^{k-a_r}) = \phi_1(p^{k-a_r})$ distinct values. Next, we pick a second edge in $L_1^{r}$ that connects to either $\lambda_1$ or $\lambda_2$; without loss of generality (relabeling vertices as needed), suppose it is $\lambda_2$. Letting $\lambda_3$ denote the other vertex of this edge, then $\lambda_3 - \lambda_2 = s_2 p^{a_r}$ for some $s_2 \in \mathbb{Z}_{p^{k-a_r}}^*$. However because $a_r$ is the maximal weight in $G$, the edge connecting $\lambda_1$ and $\lambda_3$ also has weight $a_r$. On the other hand, we have
$$\lambda_3 - \lambda_1 = (\lambda_3 - \lambda_2) + (\lambda_2 - \lambda_1) = (s_2 + s_1)p^{a_r} \text{ where } s_2 + s_1 \in \mathbb{Z}^*_{p^{k-a_r}}.$$

\noindent Hence, $s_2 \not\equiv -s_1 \bmod p^{k-{a_r}}$, and therefore there are $\phi_1(p^{k-a_r}) - p^{k-a_r-1} = \phi_2(p^{k-a_r})$ possible values for $s_2$.
Repeating this procedure, we can assign $\phi_i(p^{k-a_r})$ values to the difference of the vertices from the $i$th edge in $L_1^{r}$. Now the claim immediately follows.

\vspace{.1 in}

The preceding discussion applies to any of the linked cells of weight $a_r$, because edges in distinct linked cells never share a common vertex. Hence, we conclude that the number of possible values of edge differences in $L^{r}$ equals
$$\prod_{j=1}^{\ell(r)} \prod_{i=1}^{|L_j^{r}|} \phi_{i}(p^{k-a_r}).$$

Next, suppose that we have enumerated all edge differences from all linked cells having weight $a_{t+1}, \dots, a_r$ for some fixed $t$. We now consider linked cells for the weight $a_t$. The procedure proceeds just as before, with the only difference being that two edges of any weight lower than $a_r$ may be linked via some subtree of $T$ containing other higher weights. However this presents no new difficulties.

\vspace{.1 in}

Fix a linked cell with weight $a_t$ and choose a first edge with vertices $\lambda_{c_1}$ and $\lambda_{c_2}$. As above, this edge corresponds to one of $\phi_1(p^{k-a_t})$ possible differences between values $\lambda_{c_1}$ and $\lambda_{c_2}$. Given another edge linked to the aforementioned edge in this linked cell, it either shares or does not share a vertex with the first edge. We consider these cases separately.

\vspace{.1 in}

First, suppose the two edges share a common vertex $\lambda_{c_2}$. Then as in the previous case, the connecting edge between $\lambda_{c_1}$ and $\lambda_{c_3}$ must have weight at least $a_t$ (as this edge otherwise has weight greater than $a_t$ and such vertices have been previously considered),

and thus we can choose the value for $\lambda_{c_3} - \lambda_{c_2}$ in $\phi_2(p^{k-a_t})$ ways. 

\vspace{.1 in}

Alternately, suppose that the two edges are connected through already established edges of higher weights on the vertices $\lambda_{d_1}, \lambda_{d_2}, \dots, \lambda_{d_s}$. 
Without loss of generality, assume that the vertices $\lambda_{c_1}$ and $\lambda_{c_4}$ are the initial and terminal vertices, respectively, in this second edge.
We know that $\lambda_{c_2} - \lambda_{c_1} = rp^{k-a_t}$ and $\lambda_{c_4} - \lambda_{c_3} = r'p^{a_t}$ for some $r,r' \in \mathbb{Z}^*_{p^{k-a_t}}$. Also since the edges connecting $\lambda_{c_2}$ to $\lambda_{d_1}$, $\lambda_{d_s}$ to $\lambda_{c_3}$, and  $\lambda_{d_i}$ to $\lambda_{d_j}$ for all $1 \leq i < j \leq s$ have weights higher than $a_t$, it follows that $0 \equiv \lambda_{d_1}-\lambda_{c_2} \equiv \lambda_{c_3}-\lambda_{d_s} \equiv \lambda_{d_j}-\lambda_{d_i} \bmod{p^{a_t+1}}$ and these observations give us
\begin{align*}
\lambda_{c_4} - \lambda_{c_1} &\equiv (\lambda_{c_2} - \lambda_{c_1}) + (\lambda_{d_1} - \lambda_{c_2}) + (\lambda_{d_2} - \lambda_{d_1}) + \dots + (\lambda_{c_3} - \lambda_{d_s}) + (\lambda_{c_4} - \lambda_{c_3}) \\
&\equiv (r + r') p^{a_t} \bmod{p^{a_t+1}}. 
\end{align*}

\noindent However, by an inductive use of the Triangle Inequality, we see that the edge directly connecting $c_1$ and $c_4$ must have weight $a_t$. Thus, $r + r' \not\equiv 0 \bmod p$, and the number of permissible choices for $r'$ is therefore $p^{k-a_t}-2p^{k-a_t-1} = \phi_2(p^{k-a_t})$.

\vspace{.1 in}

Continuing this process, we can see that when we add the $i$-th edge in this linked cell (if it exists), we can find a path between it and the previous $(i-1)$ edges in $T$ sharing the same linked cell, giving $\phi_i(p^{k-a_t})$ choices for the corresponding edge differences.

\vspace{.1 in}

At this point we have considered every edge in $T$. The number of possible edge differences among all of the edges in $T$ equals
$$\prod_{t=1}^r \prod_{j=1}^{\ell(t)} \prod_{i=1}^{|L_j^{t}|} \phi_{i}(p^{k-a_t}).$$

In summary, we have specified the number of values that the differences of the vertices to each of the edges in our permissible tree can attain. Consequently, as soon as we specify the value of one vertex, in which there are $p^k$ possible choices, we have uniquely determined (by our work above) the values of the remaining vertices through their differences. Therefore, the number of possible diagonal matrices with the given valuation graph equals
$$p^k \cdot \prod_{t=1}^r \prod_{j=1}^{\ell(t)} \prod_{i=1}^{|L_j^{t}|} \phi_{i}(p^{k-a_t}).$$

\vspace{.1 in}

Finally, we note that permuting the order of the diagonal entries of any diagonal matrix associated with $G$ yields a valuation graph isomorphic to $G$. Since these correspond to the weighted graph automorphisms of $G$, dividing our last formula by $|\text{Aut}(G)|$ yields the desired enumeration formula.

\end{proof}

\noindent \textbf{Remark:} Note that the group of weighted automorphisms of $G$ is a subgroup of all automorphisms (under composition of isomorphisms) of the corresponding unweighted graph version of $G$. Since $G$ is a complete graph with $n$ vertices, we know that there are $|S_n| = n!$ unweighted graph automorphisms of $G$ (which can be represented by $n \times n$ permutation matrices). Then, Lagrange's Theorem for groups implies that $|\text{Aut}(G)| = \frac{n!}{\sigma(G)}$, where $\sigma(G) = [S_n : \text{Aut}(G)]$ denotes the number of vertex permutations yielding non-isomorphic valuation graphs from $G$. In this manner, one can determine alternatively find the value of $|\text{Aut}(G)|$ by directly computing $\sigma(G)$.


\vspace{.1 in}

So far, Theorem \ref{thm:linked} allows us to enumerate diagonal matrices with distinct diagonal entries with an associated valuation graph. The following proposition addresses how to extend this theorem to also enumerate diagonal matrices whose diagonal entries are not distinct.

\begin{prop} \label{thm:multiple}
Let $D \in M_n(\mathbb{Z}_{p^k})$ be a diagonal matrix with distinct diagonal entries $\lambda_1, \dots , \lambda_g$, and let $D' \in M_g(\mathbb{Z}_{p^k})$ be the corresponding diagonal matrix with (distinct) diagonal entries $\lambda_1, \dots , \lambda_g$. If $D$ has exactly $n_m$ distinct $m \times m$ diagonal blocks for each $m \in \{1, 2, \dots, g\}$, then $$t(T) = \frac{g!}{n_1! \dots n_g!} \cdot t(T'),$$ where $T$ and $T'$ are the types of $D$ and $D'$, respectively.
\end{prop}

\begin{proof} Since we know by hypothesis that $D$ and $D'$ share the same number of distinct diagonal entries, it suffices to count the number of ways to arrange the diagonal blocks (each of which is distinguished by a different scalar on their respective diagonals) in $D$. Since the number of ways of arranging these diagonal blocks in $D$ equals $\frac{g!}{n_1! \dots n_g!}$, the conclusion of this theorem is now an immediate consequence.
\end{proof}

Now that we have Theorem \ref{thm:linked} and Proposition \ref{thm:multiple} at our disposal, we are more than ready to enumerate the diagonalizable $n \times n$ matrices in the cases where $n = 3$ and $4$; this we address in the next two sections. Before doing this, we would like to put our theory of valuation graphs into perspective by giving an example that illustrates the theory we have developed for the valuation graph.

\vspace{.1 in}

\noindent \textbf{Example:} 
Consider the diagonal matrix $D \in M_6(\mathbb{Z}_{3^3})$ whose diagonal entries are 0, 1, 2, 4, 5, and 11. Then, its corresponding valuation graph $G$ is depicted in Figure 1 below. 



\begin{figure}[H]
\centering
\includegraphics[width = 2.3 in]{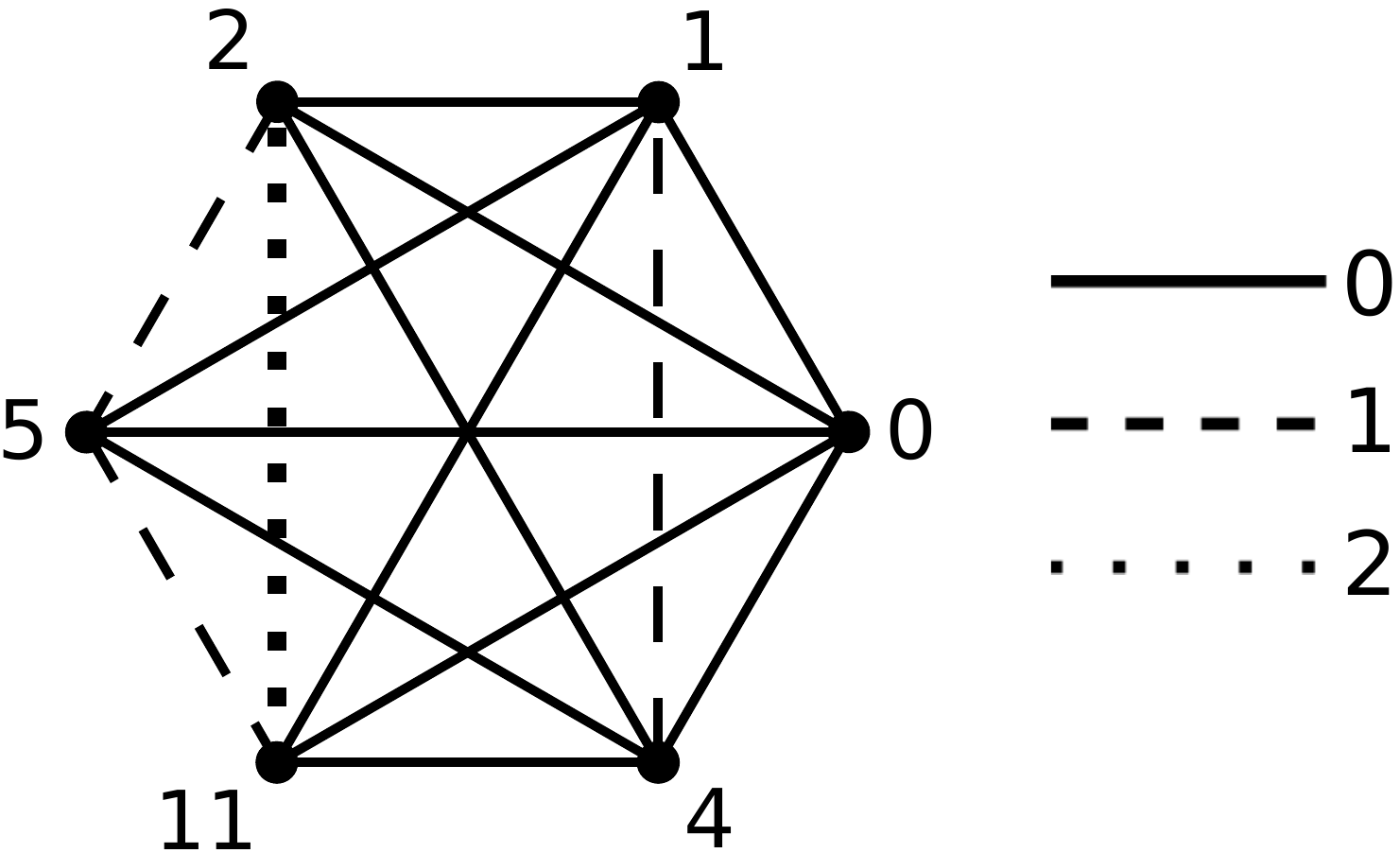}
\caption{The valuation graph $G$ corresponding to $D$.}
\end{figure}

\noindent Observe the number of distinct weights in $G$ is $3$, consistent with Lemma \ref{thm:number_of_weights}, and that the highest edge weight is $2$.

\vspace{.1 in}

Next, we give examples of permissible spanning trees for $G$ and partition their edges into linked cells. Figure 2 shows three permissible spanning trees $T_1,T_2,T_3$ for $G$ and their linked cells $L_1^1, L_1^2, L_2^2$, and $L_1^3$. 

\begin{figure}[H]
\centering
\includegraphics[width = 3 in]{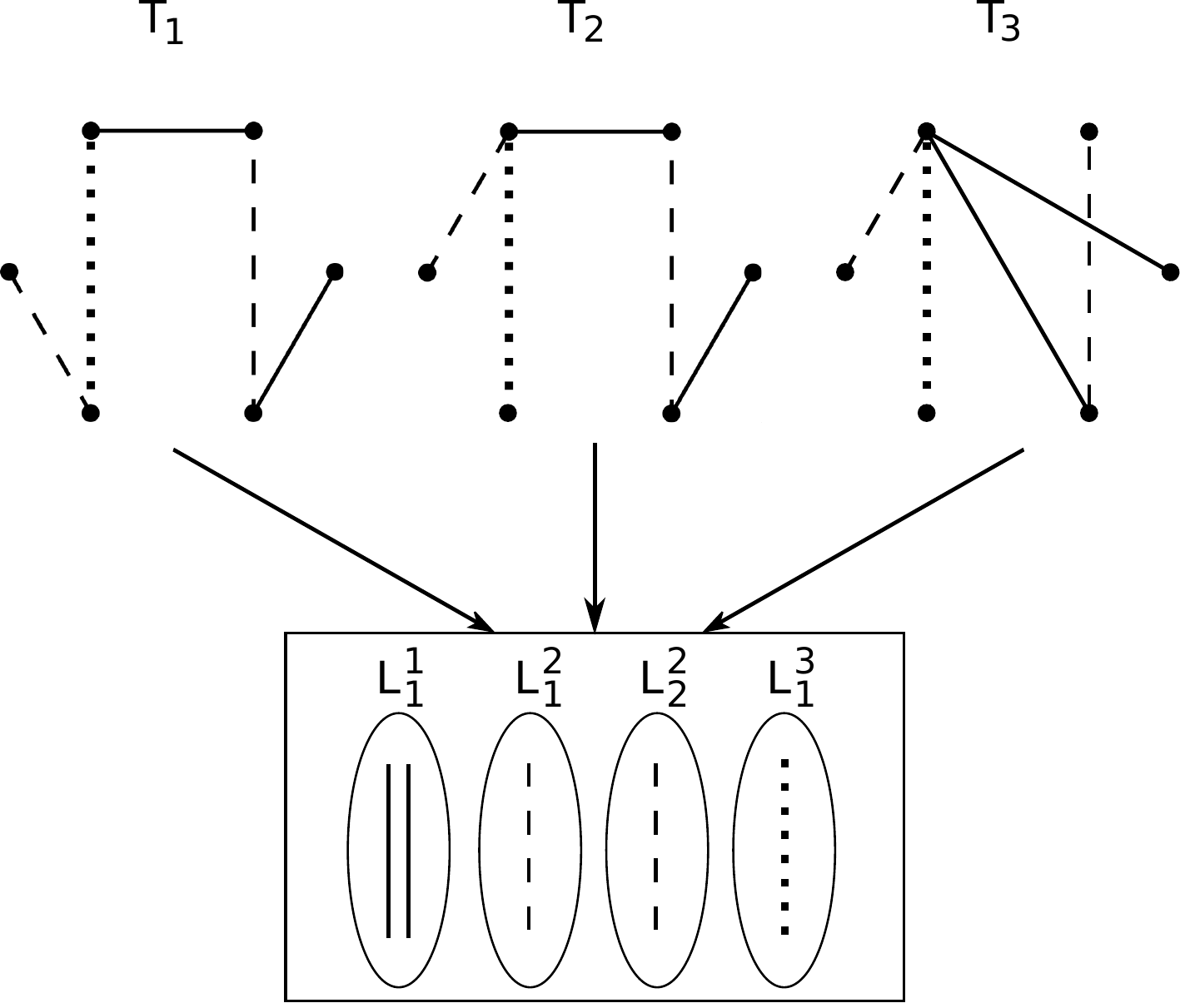}
\caption{Three permissible spanning trees for $G$ and their linked cells.}
\end{figure}

Although each of these spanning trees have different degrees, they all have the same edge decomposition into linked cells. Thus, we can use any of these permissible spanning trees to enumerate the number of similarity classes of diagonal matrices sharing $G$ as its valuation graph. To this end, it remains to compute $|\text{Aut}(G)|$. Since we can permute the vertices $2$ and $11$, as well as the vertices $1$ and $4$ without altering $G$, this implies that $|\text{Aut}(G)| = 2!\cdot2!$. Therefore by Theorem \ref{thm:linked}, the number of similarity classes of diagonal matrices with valuation graph $G$ equals

\begin{align*}
    \frac{3^3}{2! \cdot 2!} \cdot \prod_{t=0}^2 \prod_{j=1}^{\ell(t)} \prod_{i=1}^{|L_j^{t}|} \phi_{i}(3^{3-t})
    &= \frac{27}{4} \cdot\phi_1(3^3) \cdot \phi_2(3^3) \cdot \phi_1(3^2) \cdot \phi_1(3^2) \cdot \phi_1(3^1)\\
    &= 78732.
\end{align*}

\section{Enumerating the \texorpdfstring{$3 \times 3$}{TEXT} Diagonalizable Matrices}

\begin{theorem}
The number of $3 \times 3$ matrices with entries in $\mathbb{Z}_{p^k}$ that are diagonalizable over $\mathbb{Z}_{p^k}$ is
\begin{align*}
|\emph{Diag}_3(\mathbb{Z}_{p^k})| &= p^k + \frac{p^{k+2}(p^3-1)(p^{5k}-1)}{p^5 - 1} + \frac{p^{k+3}(p^3-1)(p-2)(p+1)(p^{8k}-1)}{6(p^8 - 1)}\\
&+ \frac{p^{k+3}(p^2-1)}{2}\Bigg( \frac{p^{8k}-p^8}{p^8-1} - \frac{p^{5k}-p^5}{p^5-1}\Bigg).
\end{align*}
\end{theorem}

\begin{proof}
We first enumerate all of the $3 \times 3$ diagonal matrix types. There are three partitions of $3$, namely $3$, $2+1$, and $1+1+1$. The trivial partition yields the type of scalar matrices 
$$T_1 = \left \{ \begin{pmatrix}
\lambda &&\\
& \lambda&\\
&& \lambda\\
\end{pmatrix} \; : \; \lambda \in \mathbb{Z}_{p^k} \right\}.$$

\noindent As with the type of $2 \times 2$ scalar diagonal matrices, we have $t(T_1) = p^k$ and $c(T_1) = |GL_3(\mathbb{Z}_{p^k})|$.

\vspace{.1 in}
 
The partition $3 = 2+1$ comprises $k$ distinct types as $i \in \{0, 1, \dots , k-1\}$:
$$T_2^{(i)} = \left\{\begin{pmatrix} \lambda_1 &&\\ & \lambda_1&\\ && \lambda_2\\ \end{pmatrix} \; : \; p^i \; || \; (\lambda_1-\lambda_2) \right\}.$$

\noindent Proposition \ref{thm:multiple} relates these types to the non-scalar types of $2 \times 2$ diagonal matrices, and thus $$t(T_2^{(i)}) = \frac{2!}{1!1!} \cdot \frac{p^k \phi(p^{k-i})}{2!} = p^k \phi(p^{k-i}).$$

\noindent Next, Proposition \ref{thm:centralizer}  gives us
$c(T_2^{(i)}) = \phi(p^k) \cdot \vert GL_2(\mathbb{Z}_{p^k}) \vert \cdot p^{4i}$.

\vspace{.1 in}

Finally, the partition $3=1+1+1$ comprises two distinct classes of diagonal matrix types that we concisely give by their respective valuation graphs in the figure below:

\begin{figure}[H]
\centering
\includegraphics[width = 3
in]{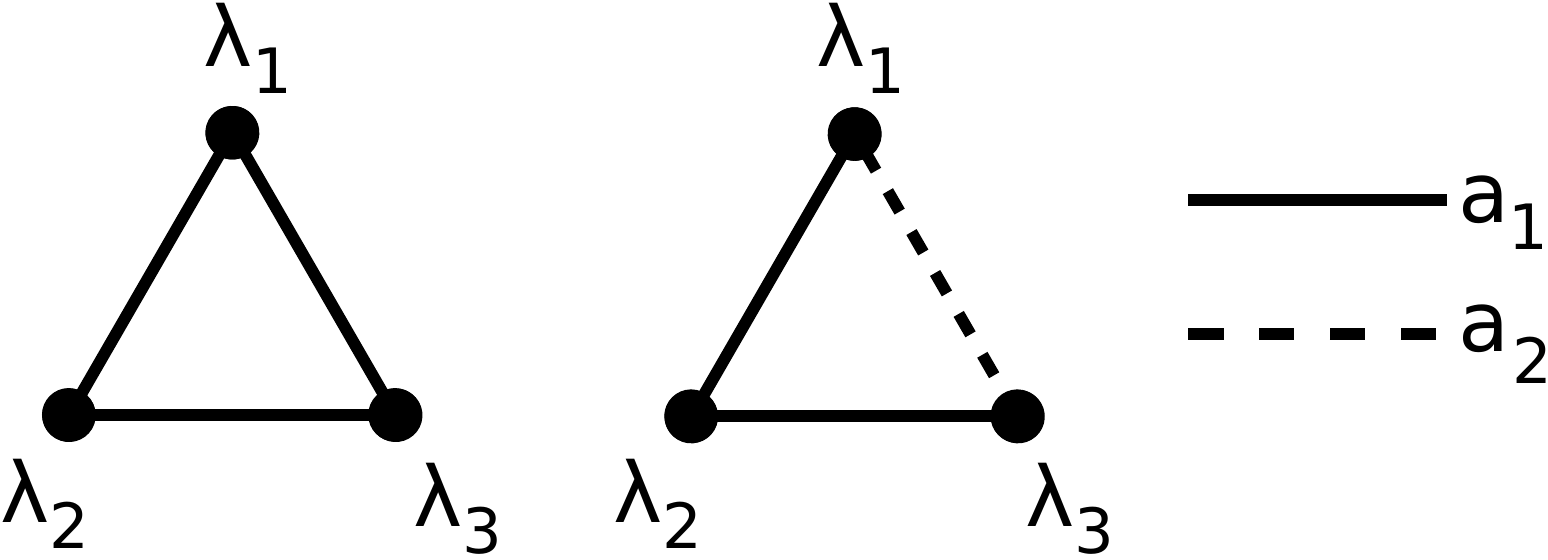}
\caption{Two valuation graph classes in the $3 \times 3$ case.}
\end{figure}

For the first valuation graph, let $i \in \{0, 1, \dots, k-1\}$ denote the common weight of the three edges on the first valuation graph given above. Letting $T_{3a}^{(i)}$ denote this type,  Theorem \ref{thm:linked} yields $t(T_{3a}^{(i)})= \displaystyle \frac{p^k \phi (p^{k-i}) \phi_2(p^{k-i})}{3!}$, and Proposition \ref{thm:centralizer} gives us $c(T_{3a}^{(i)}) = \phi (p^k)^3 p^{6i}$. 

\vspace{.1 in}

For the second valuation graph, let $i$ and $j$ denote the weights in the second valuation graph given above; note that $i \in \{0, \dots, k-2\}$ and $j \in \{i+1, \dots, k-1\}$. Letting $T_{3b}^{(i,j)}$ denote this type, Theorem \ref{thm:linked}, gives us $t(T_{3b}^{(i,j)}) = \displaystyle \frac{p^k \phi (p^{k-i})\phi (p^{k-j})}{2!}$, and Proposition \ref{thm:centralizer} yields $c(T_{3b}^{(i, j)}) = \phi (p^k)^3 p^{4i + 2j}$. 


\vspace{.1 in}

Finally, we use (\ref{eq:2}) to enumerate the $3 \times 3$ diagonal matrices and conclude that
\begin{align*}
\vert\text{Diag}_3(\mathbb{Z}_{p^k})\vert &= 
p^k + \frac{p^{k+2}(p^3-1)(p^{5k}-1)}{p^5 - 1} + \frac{p^{k+3}(p^3-1)(p-2)(p+1)(p^{8k}-1)}{6(p^8 - 1)} \\
&+ \; \frac{p^{k+3}(p^2-1)}{2}\Bigg( \frac{p^{8k}-p^8}{p^8-1} - \frac{p^{5k}-p^5}{p^5-1}\Bigg).
\end{align*}
\end{proof}

\section{Enumerating the \texorpdfstring{$4 \times 4$}{TEXT} Diagonalizable Matrices}







We first address the $4 \times 4$ diagonal matrices with repeated diagonal entries. By using Propositions \ref{thm:centralizer} and \ref{thm:multiple}, we obtain the results in the following tables. Table 1 deals with the cases where there are at most two distinct diagonal entries.

\begin{table}[ht]
\centering
\begin{tabular}{|c|Sc|c|c|}
\hline   
Type $T$ & Valuation Graph & $t(T)$ & $c(T)$\\
\hline 
$\begin{pmatrix}
\lambda &&&\\
& \lambda&&\\
&& \lambda&\\
&&& \lambda\\
\end{pmatrix}$ & \includegraphics[width = 0.7 in, valign=c] {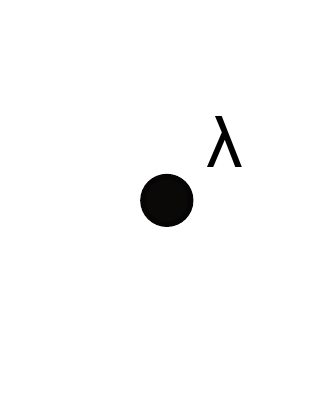} & $p^k$ & $|GL_4(\mathbb{Z}_{p^k})|$ \\

\hline 

$\begin{pmatrix} \lambda_1 &&&\\ & \lambda_1 &&\\
&& \lambda_1 &\\ &&& \lambda_2\\ \end{pmatrix}$  & \includegraphics[width = 1.0 in, valign=c] {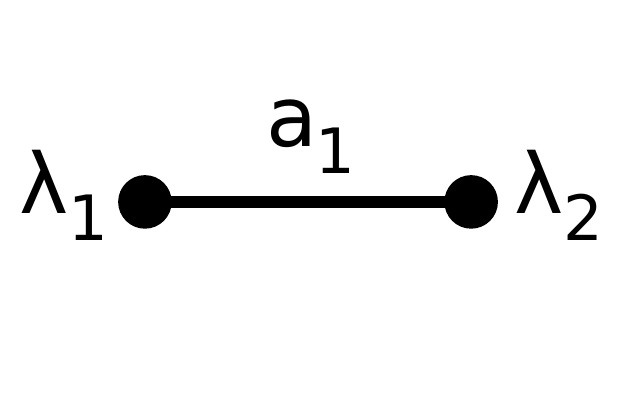} & $p^k \phi(p^{k-i})$ & $ p^{6i} \phi(p^k) \, |GL_3(\mathbb{Z}_{p^k})|$ \\

\hline 

$\begin{pmatrix}
\lambda_1 &&&\\
& \lambda_1 &&\\
&& \lambda_2  &\\
&&& \lambda_2\\
\end{pmatrix}$ & \includegraphics[width = 1.0 in, valign=c] {k2.pdf} & $\displaystyle\frac{p^k \phi(p^{k-i})}{2}$ & $p^{8i} \, |GL_2(\mathbb{Z}_{p^k})|^2 $ \\

\hline    
\end{tabular}
\caption{$4 \times 4$ diagonal matrix types with at most two distinct diagonal entries.}
\end{table}

In Table 2, we consider the more involved case where a given diagonal matrix has three distinct diagonal entries. 
\begin{table}[ht]
\centering
\begin{tabular}{|c|Sc|c|c|}
\hline   
Type $T$ & Valuation Graph & $t(T)$ & $c(T)$\\
\hline 

$\begin{pmatrix}
\lambda_1 &&&\\
& \lambda_1 &&\\
&& \lambda_2 &\\
&&& \lambda_3\\
\end{pmatrix}$ & 
\includegraphics[width = 1.0in, valign=c] {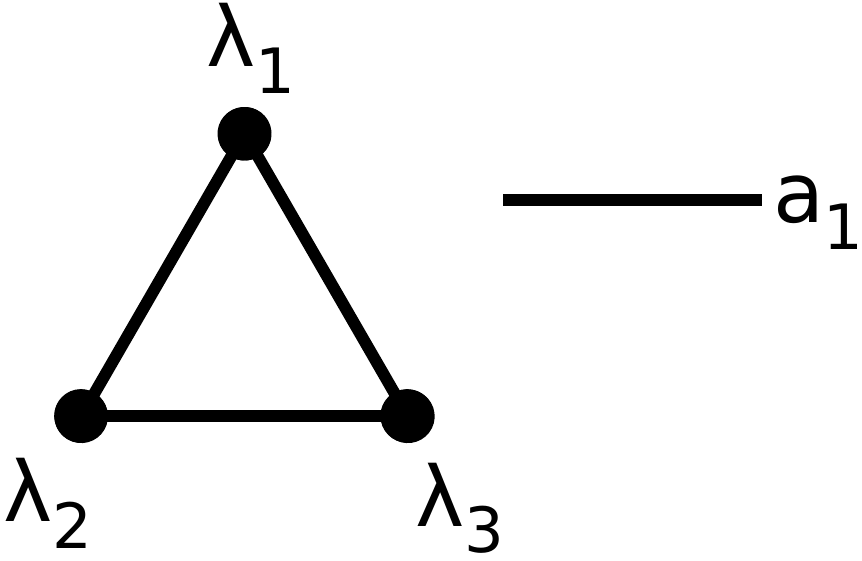}
& $\displaystyle\frac{p^k \phi(p^{k-i}) \phi_2(p^{k-i})}{2}$ & $p^{10i} \phi(p^k)^2 \, |GL_2(\mathbb{Z}_{p^k})|$ \\

\hline 

$\begin{pmatrix}
\lambda_1 &&&\\
& \lambda_1 &&\\
&& \lambda_2 &\\
&&& \lambda_3\\
\end{pmatrix}$ & 
\includegraphics[width = 1.0in, valign=c] {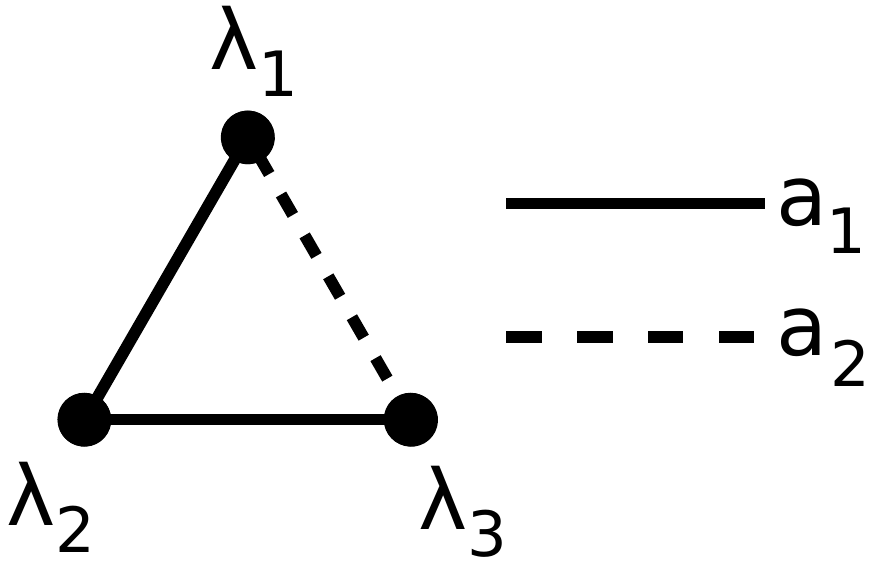}
&$\displaystyle\frac{3p^k \phi(p^{k-i}) \phi(p^{k-j})}{2}$ & $p^{6i+4j} \phi(p^k)^2 \, |GL_2(\mathbb{Z}_{p^k})|$ \\

\hline

$\begin{pmatrix}
\lambda_1 &&&\\
& \lambda_1 &&\\
&& \lambda_2 &\\
&&& \lambda_3\\ 
\end{pmatrix}$ & 
\includegraphics[width = 1.0in, valign=c] {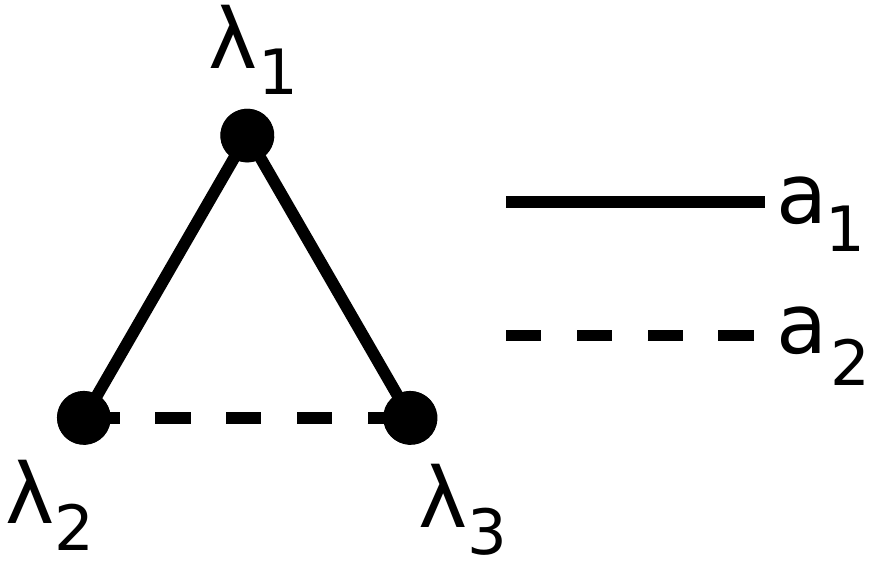}
&$\displaystyle\frac{3p^k \phi(p^{k-i}) \phi(p^{k-j})}{2}$ & $p^{8i+2j} \phi(p^k)^2 \, |GL_2(\mathbb{Z}_{p^k})|$ \\

\hline    
\end{tabular}
\caption{$4 \times 4$ diagonal matrix types with three distinct diagonal entries.}
\end{table}

It remains to enumerate the diagonal matrix types where the diagonal entries are distinct. By inspection, we find that there are 6 distinct classes of valuation graphs if we disregard the actual weights of their edges. We summarize the pertinent information for each of these six valuation graphs in Table 3.






\begin{table}[H]
\centering
\begin{tabular}{|Sc|c|c|}
\hline   
Valuation Graph & $t(T)$ & $c(T)$\\
\hline 

\includegraphics[width = 1.2in, valign=c] {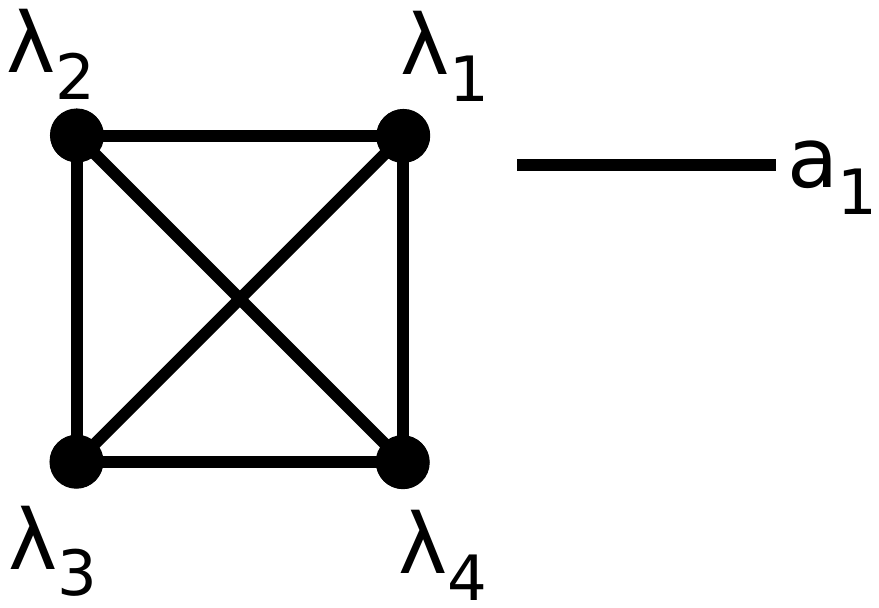}
& $\displaystyle\frac{p^k \phi(p^{k-i}) \phi_2(p^{k-i}) \phi_3(p^{k-i})}{4!}$ & $p^{12i} \phi(p^k)^4$\\

\hline 

\includegraphics[width = 1.2in, valign=c] {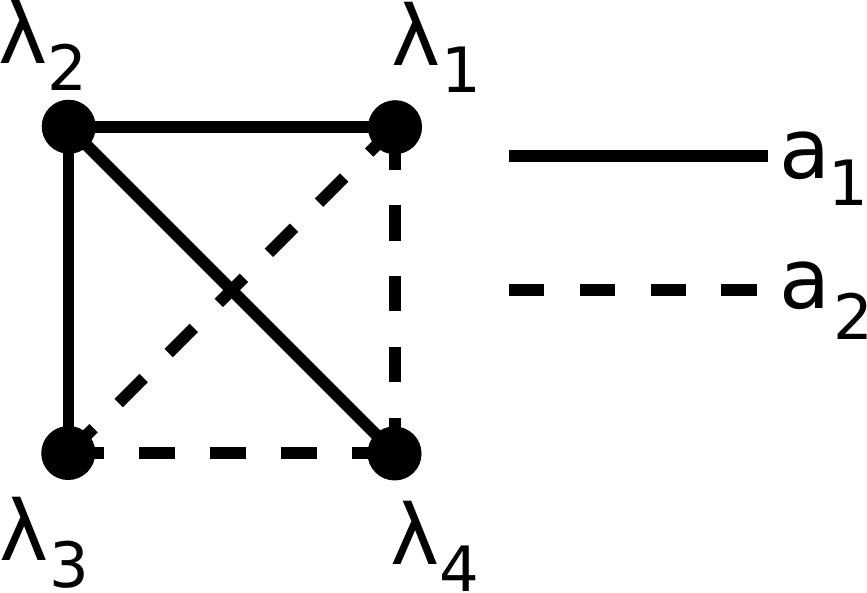}
&$\displaystyle\binom{4}{3}\frac{p^k \phi(p^{k-i}) \phi(p^{k-j}) \phi_2(p^{k-j})}{4!}$ & $p^{6i+6j} \phi(p^k)^4$\\

\hline

\includegraphics[width = 1.2in, valign=c] {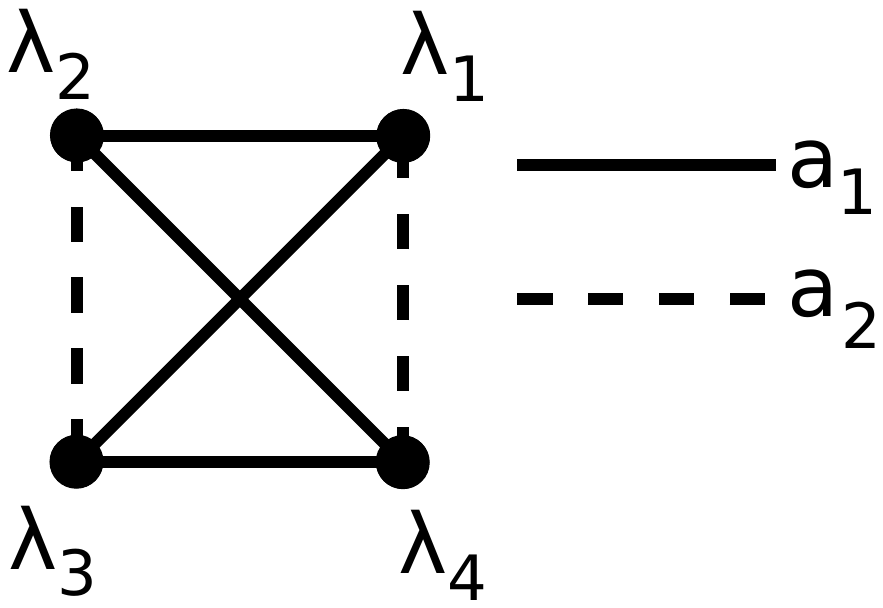}
&$\displaystyle\frac{1}{2}\binom{4}{2} \frac{p^k \phi(p^{k-i}) \phi(p^{k-j})^2}{4!}$ & $p^{8i+4j} \phi(p^k)^4$\\

\hline    

\includegraphics[width = 1.2in, valign=c] {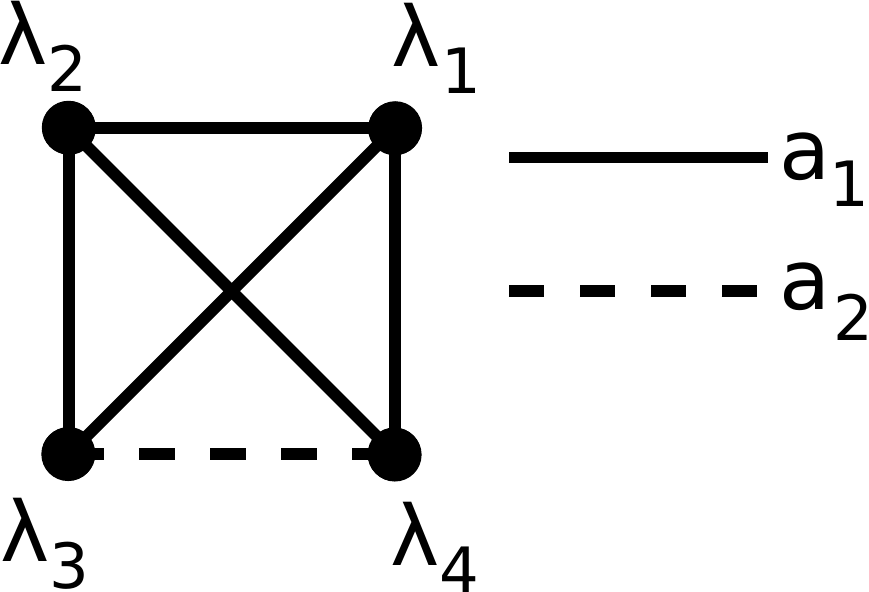}
&$\displaystyle\binom{4}{2} \frac{p^k \phi(p^{k-i}) \phi_2(p^{k-i}) \phi(p^{k-j})}{4!}$ & $p^{10i+2j} \phi(p^k)^4$\\

\hline    

\includegraphics[width = 1.2in, valign=c] {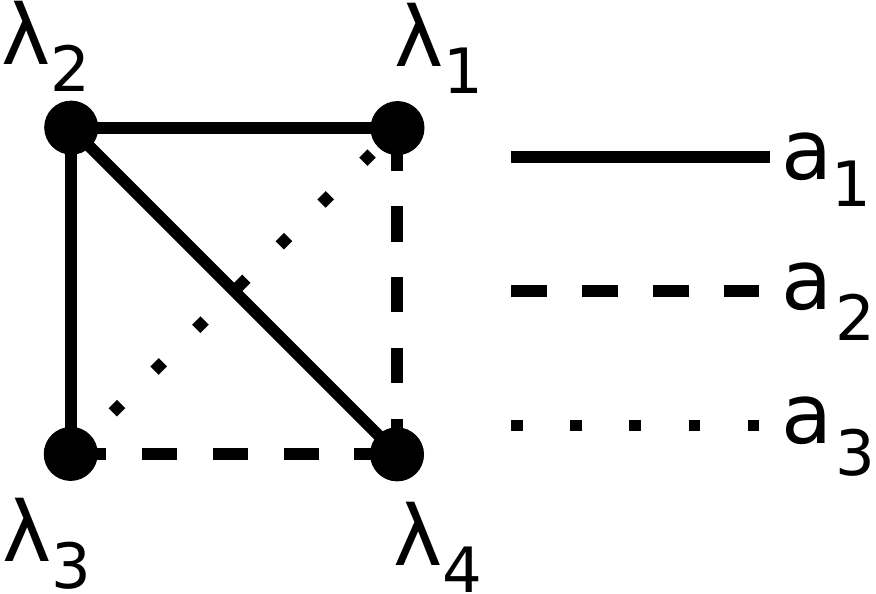}
&$\displaystyle\binom{4}{3} \binom{3}{1} \frac{p^k \phi(p^{k-i}) \phi(p^{k-j}) \phi(p^{k-m})}{4!}$ & $p^{6i+4j+2m} \phi(p^k)^4$\\

\hline    

\includegraphics[width = 1.2in, valign=c] {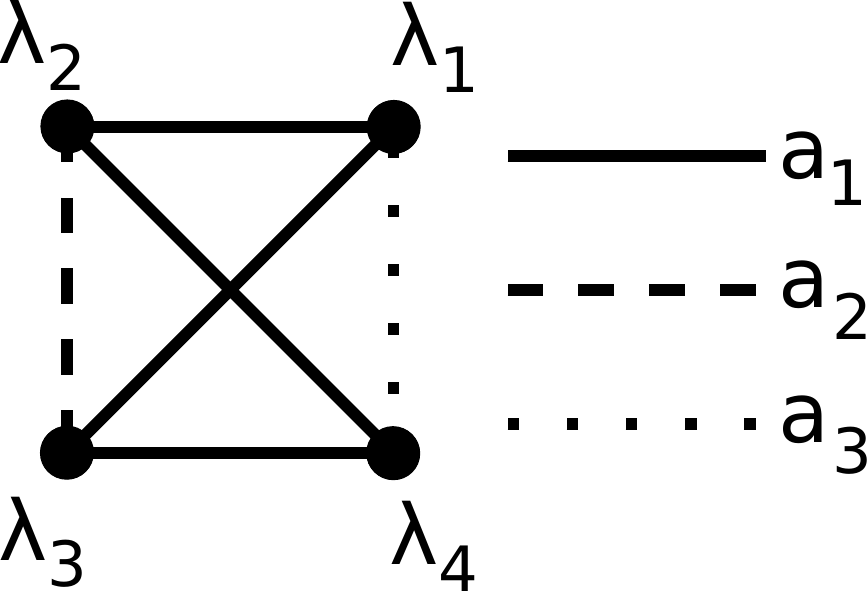}
&$\displaystyle\binom{4}{4} \binom{4}{2} \frac{p^k \phi(p^{k-i}) \phi(p^{k-j}) \phi(p^{k-m})}{4!}$ & $p^{8i+2j+2m} \phi(p^k)^4$\\

\hline    
\end{tabular}
\caption{$4 \times 4$ diagonal matrix types with distinct diagonal entries.}  
\end{table}

By using (\ref{eq:2}), one can now find the number of $4 \times 4$ diagonalizable matrices over $\mathbb{Z}_{p^k}$. In light of the many cases from the three tables above, the final formula will be quite long and messy to explicitly write out, and we therefore have chosen not to include it here (although the curious reader should have no problem constructing it if necessary).










\section{The Proportion of Diagonalizable Matrices Over \texorpdfstring{$\mathbb{Z}_{p^k}$}{TEXT}}

Kaylor \cite{Kaylor} noted that as the size of the field $\mathbb{F}_q$ increases, the proportion of matrices in
$M_n(\mathbb{F}_q)$ with all eigenvalues in $\mathbb{F}_q$ approaches $\frac{1}{n!}$; that is,
$$\lim_{q \to \infty} \frac{|\text{Eig}_n(\mathbb{F}_q)|}{|M_n(\mathbb{F}_q)|} = \frac{1}{n!}.$$

\noindent In particular, the work in \cite{Kaylor} also implies that as the size of $\mathbb{F}_q$ increases, the proportion of matrices in $M_n(\mathbb{F}_q)$ that are diagonalizable over $\mathbb{F}_q$ approaches $\frac{1}{n!}$ as well. We generalize this latter result by replacing $\mathbb{F}_q$ in the case of $q = p$ with $\mathbb{Z}_{p^k}$.

\begin{theorem}
Fix positive integers $n$ and $k$, and let $p$ be a prime number. Then,
$$\lim_{p \to \infty} \frac{|\emph{Diag}_n(\mathbb{Z}_{p^k})|}{|M_n(\mathbb{Z}_{p^k})|} = \frac{1}{n!}.$$
\end{theorem}

\begin{proof}
Letting $i$ index the distinct types of diagonal matrices, we let $T_{n,i}$ denote the $i$-th distinct type of a diagonal matrix in $M_n(\mathbb{Z}_{p^k})$. Note that we can view $|\text{Diag}_n(\mathbb{Z}_{p^k})|$ as a polynomial in powers of $p$. Since we are taking a limit as $p \to \infty$, it suffices to determine which diagonal matrix types contributes to the leading term of $|\text{Diag}_n(\mathbb{Z}_{p^k})|$. We accomplish this by first computing its degree. 
\begin{align*}
\deg{|\text{Diag}_n(\mathbb{Z}_{p^k})|} &= \deg \Big(\sum_{i = 1}^{|\mathcal{T}(n)|} t(T_{n,i}) s(T_{n,i})\Big)\\
&= \max_{1 \leq i \leq |\mathcal{T}(n)|} \deg\Big(t(T_{n,i}) s(T_{n,i})\Big)\\
&= \max_{1 \leq i \leq |\mathcal{T}(n)|} \deg \Big( t(T_{n,i}) \frac{|GL_n(\mathbb{Z}_{p^k})|}{c(T_{n,i})}\Big)\\
&= \max_{1 \leq i \leq |\mathcal{T}(n)|} (\deg{|GL_n(\mathbb{Z}_{p^k})|} + \deg{t(T_{n,i})} - \deg{c(T_{n,i})})\\
&= \max_{1 \leq i \leq |\mathcal{T}(n)|} (kn^2 + \deg{t(T_{n,i})} - \deg{c(T_{n,i})}).
\end{align*}

\noindent By Proposition \ref{thm:centralizer}, we find that
\begin{align*}
\deg c(T_{u,i}) &= \sum_{i = 1}^r \deg |GL_{m_i}(\mathbb{Z}_{p^k})| + \sum_{1 \leq i < j \leq k} \deg{p^{2m_im_jl_{ij}}}\\ 
&= k \sum_{i=1}^r m_i^2 + \sum_{1 \leq i < j \leq k} 2 m_i m_j l_{ij}\\
&\geq  k \sum_{i=1}^r m_i^2 \text{, since each } l_{ij} \geq 0\\
&\geq kn \text{, since each } m_i \geq 1. 
\end{align*}

\noindent Moreover, Theorem \ref{thm:linked} yields
\begin{align*}
\deg t(T_{n,i}) &= k + \sum_{t=1}^r \sum_{j=1}^{J_t} \sum_{i=1}^{|L_j^{(a_t)}|} (k - a_t)\\
&\leq k + \sum_{t=1}^r \sum_{j=1}^{J_t} k |L_j^{(a_t)}|\\
&= k + k(n-1)\\
&= kn.
\end{align*}

\noindent Therefore $\deg{t(T_{n,i})} \leq \deg{c(T_{n,i})}$, with equality occurring if and only if the diagonal matrix type in which its diagonal entries are distinct and their differences are units in $\mathbb{Z}_{p^k}$. Hence, $\deg{\vert \text{Diag}_n (\mathbb{Z}_{p^k})\vert} = kn^2$, and using the aforementioned diagonal matrix type, the leading coefficient of $|\text{Diag}_n(\mathbb{Z}_{p^k})|$ equals $\frac{1}{n!}$ by Theorem \ref{thm:linked}. Thus, we have 
$$\frac{|\text{Diag}_n(\mathbb{Z}_{p^k})|}{|M_n(\mathbb{Z}_{p^k})|} = \frac{\frac{1}{n!} \, p^{kn^2} + O(p^{kn^2 - 1})}{p^{kn^2}}.$$ 
The desired limit immediately follows by letting $p \to \infty$.
\end{proof}

\section{Future Research}

As we have seen, given a ring of the form $\mathbb{Z}_{p^k}$ and a positive integer $n$, we have given a procedure to compute $|\text{Diag}_n(\mathbb{Z}_{p^k})|$. The main difficulty that remains is enumerating the possible valuation graph classes (up to automorphism and disregarding the actual values of the weights) corresponding to $n \times n$ diagonal matrices. As demonstrated in the previous sections, it suffices to enumerate such classes corresponding to $n \times n$ diagonal matrices with distinct diagonal entries; let $a_n$ denote this quantity. We have seen that $a_2 = 3$ and $a_4 = 6$, and it turns out that $a_5 = 20$ (see Figure 4 below for these classes).  It would be of interest to find at least a recursive formula that determines $a_n$ for a given value of $n$. 

\begin{figure}[ht]
\begin{center}
\includegraphics[width = 5.5in]{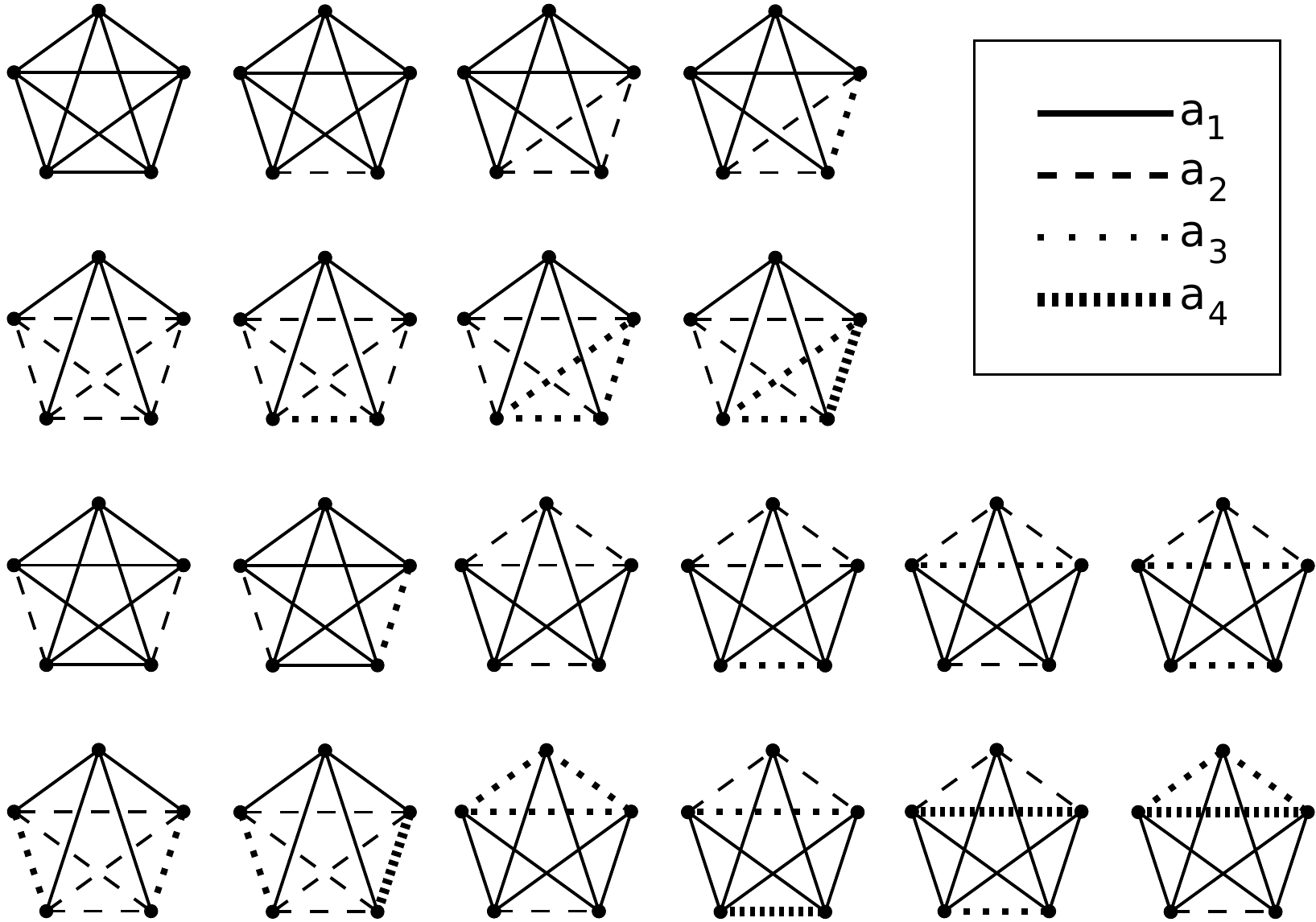}
\caption{The twenty $5 \times 5$ valuation graph classes.}
\end{center}
\end{figure}

In addition to this, it would be of interest to extend our work to include matrices with Jordan Canonical Forms (JCFs) over $\mathbb{Z}_{p^k}$; that is matrices similar to a block diagonal matrix comprised of the Jordan matrices
$$\begin{pmatrix}
\lambda &  1  &  0  & \dots &  0 & 0\\
0 & \lambda &  1  &  \dots  &  0 & 0\\
0 & 0 & \lambda & \dots  &  0 & 0\\
\vdots & \vdots &  \vdots & \ddots  &  \vdots  &  \vdots\\
0 & 0 &  0 & \dots & \lambda & 1\\
0 & 0 & 0 &  \dots  &  0 &  \lambda \end{pmatrix}$$

\noindent for some $\lambda \in \mathbb{Z}_{p^k}$. One would have to be careful performing such an enumeration, because it is possible for a given matrix to have more than one distinct JCF over $\mathbb{Z}_{p^k}$. For instance in $\mathbb{Z}_4$, we have that
$$\begin{pmatrix} 0 &  1 \\ 0 &  0 \end{pmatrix} = 
\begin{pmatrix} 1 &  0  \\ 2 &  1 \end{pmatrix}
\begin{pmatrix} 2 &  1  \\ 0 &  2 \end{pmatrix}
\begin{pmatrix} 1 &  0  \\ 2 &  1 \end{pmatrix}^{-1}.$$

\noindent Although finding an enumeration formula for the centralizer of a Jordan matrix should be straightforward, this is not expected to be the case for an arbitrarily chosen JCF. 

\vspace{.1 in}

As a final remark, besides the potential non-uniqueness of a JCF, there is a reason why we have not enumerated $|\text{Eig}_n(\mathbb{Z}_{p^k})|$. Unlike in the finite field case where any matrix in $\text{Eig}_n(\mathbb{F}_q)$ has a JCF (see \cite{Kaylor} for more details), this is not even necessarily the case in $\text{Eig}_n(\mathbb{Z}_{p^k})$. For example in $\mathbb{Z}_{p^2}$, any matrix of the form $\begin{pmatrix}  \lambda & p \\  0 & \lambda
\end{pmatrix}$ has double eigenvalue $\lambda$, but lacks a Jordan Canonical Form over $\mathbb{Z}_{p^2}$. Determining all similarity classes of matrices such matrices is in general still an open question.

\end{document}